\documentclass[12pt,reqno]{amsart}
\usepackage{hyperref}
\usepackage{amsmath,amssymb}
\usepackage[english]{babel}
\usepackage{caption}
\usepackage{graphicx}
\usepackage{float}
\usepackage{fullpage}
\usepackage{pgf,tikz}
\numberwithin{equation}{section}
\tikzstyle{vertex}=[circle,draw, inner sep=0pt, minimum size=1.5pt]
\newcommand{\vertex}{\node[vertex]}

\newcommand \cochord{\operatorname{co-chord}}
\newcommand \reg{\operatorname{reg}}

\newcommand \Tor{\operatorname{Tor}}
\newcommand \ab{\nu}
\newcommand \ca{\operatorname{c}}
\newcommand \ba{\operatorname{b}}

\newtheorem{theorem}{Theorem}[section]
\newtheorem{definition}[theorem]{Definition}

\newtheorem{proposition}[theorem]{Proposition}
\newtheorem{example}[theorem]{Example}
\newtheorem{obs}[theorem]{Observation}
\newtheorem{question}[theorem]{Question}
\newtheorem{remark}[theorem]{Remark}
\newtheorem{corollary}[theorem]{Corollary}

\begin{document}
\title[Regularity of powers of bipartite graphs]{Regularity of powers of bipartite graphs}
\author{A. V. Jayanthan}
\email{jayanav@iitm.ac.in}
\author{N. Narayanan}
\email{naru@iitm.ac.in}
\author{S. Selvaraja}
\email{selva.y2s@gmail.com}
\address{Department of Mathematics, Indian Institute of Technology
Madras, Chennai, INDIA - 60036}

\maketitle

\begin{abstract}
Let $G$ be a finite simple graph and $I(G)$ denote the corresponding
edge ideal. For all $s \geq 1$, we obtain upper bounds for
$\reg(I(G)^s)$ for bipartite graphs. We then compare the properties of
$G$ and $G'$, where $G'$ is the graph associated with the
polarization of the ideal $(I(G)^{s+1} : e_1\cdots e_s)$, where
$e_1,\ldots e_s$ are edges of $G$. Using these results, we explicitly
compute $\reg(I(G)^s)$ for several subclasses of bipartite graphs.
\end{abstract}

\section{Introduction}
Let $G = (V(G) , E(G))$ denote a finite simple undirected graph with vertices
$V(G) = \{x_1,\ldots,x_n\}$ and edge set $E(G)$. By identifying the vertices with the variables in the polynomial ring 
$k[x_1,\ldots, x_n],$ where $k$ is a field, we can associate to
each graph $G$ a monomial ideal $I(G)$ generated by the set 
$\{x_i x_j \mid \{x_i , x_j \}  \in E (G) \}$.
The ideal $I(G)$ is called the \textit{edge ideal} of $G$. This notion
was introduced by Villarreal in \cite{vill_cohen}.
Since then, the researchers have been investigating the connection
between the combinatorial properties of the graphs and the algebraic
properties of the corresponding edge ideals.
In particular, there have been active research on bounding the
homological invariants of edge ideals in terms of the
combinatorial invariants of the associated graphs, see for example 
\cite{banerjee}, \cite{huneke}, \cite{froberg}, \cite{ha_adam},
\cite{sean_thesis}, \cite{khosh_moradi}, \cite{kummini}, \cite{mohammad},
\cite{nevo_peeva}, \cite{adam}, \cite{russ}, \cite{Zheng}. 
In this article, we study the
Castelnuovo-Mumford regularity of powers of edge ideals of bipartite graphs. For a
homogeneous ideal $I$, we denote by $\reg(I)$, the Castelnuovo-Mumford
regularity, henceforth called regularity, of $I$.

It was proved by Cutkosky, Herzog and Trung, \cite{CHT},
and independently by Kodiyalam \cite{vijay}, that for a
homogeneous ideal $I$ in a polynomial ring, $\reg(I^s)$ is
a linear function for $s \gg 0$, i.e., there exist
integers $a$, $b$, $s_0$ such that  $$\reg(I^s) = as + b \text{ for
all } s \geq s_0.$$ It is known that $a$ is bounded above by the
maximum of degree of elements in a minimal generating set of $I$. But
a general bound for $b$ as well as $s_0$ is unknown.  
In this paper, we consider $I=I(G),$ the edge ideal of $G$.  In this
case, there exist integers $b$ and $s_0$ such that $\reg(I^s) = 2s +
b$ for all $s \geq s_0$.  Our
objective in this paper is to find $b$ and $s_0$ in terms of
combinatorial invariants of the graph $G$. We refer the reader to
\cite{selvi_ha} for a review of results in the literature which
identify classes of edge ideals for which $b$ and $s_0$ are explicitly
computed.

It is known that for any graph $G$,
\begin{equation}\label{ag_reg_chg}
\ab(G) + 1 \leq \reg(I(G)) \leq\cochord(G) + 1, 
\end{equation}
where $\ab(G)$ denote the induced matching number of
$G$ and $\cochord(G)$ denote the co-chordal cover number of $G$.
The lower bound was proved by Katzman, \cite{kat} and the upper bound
was proved by Woodroofe, \cite{russ}. Beyarslan, H\`{a} and Trung
proved that for
any graph $G$ and $s \geq 1$, $2s+\ab(G)-1 \leq \reg(I(G))^s$, \cite{selvi_ha}. They also proved that the equality holds for
edge ideals of forests (for all $s \geq 1$) and cycles (for all $s \geq
2$). 
Moghimian, Sayed and Yassemi have shown that
the equality holds for edge ideals of whiskered cycles as well,
\cite{moghimian_seyed_yassemi}.

There is no general upper bound known for $\reg(I(G)^s)$. Woodroofe's
inequality, (\ref{ag_reg_chg}), suggests $\reg(I(G)^s) \leq
2s+\cochord(G) - 1$ for all $s \geq 1$. We prove this inequality for
bipartite graphs. Using this 
we discover new classes of graphs for which $b$ and $s_0$ can be
computed explicitly.  We determine several classes of graphs for which
the equality $$\reg(I(G)^s)=2s+b$$ holds, for every $s \geq 1$, i.e.,
$s_0=1$ and $b$ is explicitly described using combinatorial invariants
associated with $G$. One of the central ideas in our proofs is the comparison
of certain properties of a graph $G$ with those of another associated graph
$G'$. Let $G$ be a graph and $e_1, \ldots,
e_s$ be edges (not necessarily distinct) of $G$, $s \geq 1$. Banerjee, in
\cite{banerjee}, introduced the notion of \textit{even-connection} with
respect to the $s$-fold product $e_1\cdots e_s$, see Definition
\ref{even_connected}. He showed
that $(I(G)^{s+1} : e_1\cdots e_s)$ is a quadratic monomial ideal,
and hence its polarization corresponds to a graph, say $G'$. 
Then Banerjee showed that $G'$ is the union of $G$ with all the
even-connections with respect to the $s$-fold product $e_1\cdots
e_s$. Though $G'$ has been obtained from $G$ through an
algebraic operation, some of the combinatorial properties
seem to be comparable. 
The ideal $I(G')$ has emerged as a good tool in the study of
asymptotic regularity of the edge ideals, see \cite{banerjee1}, \cite{banerjee},
\cite{selvi_ha}, \cite{mohammad}. The comparison between $G$ and $G'$
(equivalently, between $I(G)$ and $I(G')$)
provides a tool to compute upper bound for the regularity.
If $G$ is an arbitrary graph, $e$ is an edge
in $G$ and $G'$ is the graph associated with the polarization of
$(I(G)^2 : e)$, then one of the main results states
that $\cochord(G') \leq \cochord(G)$ (Theorem \ref{cochord}). 
Alilooee and Banerjee proved that if $G$ is
bipartite, then so is $G'$, \cite{banerjee1}. 
Also, Banerjee proved that $\reg(I(G)^{s+1}) \leq \max
\{\reg(I(G')+2s, \reg(I(G)^s\}$, \cite{banerjee}. 
We use these results to get an upper bound for the regularity of
$I(G)^s$ when $G$ is a bipartite graph:

\begin{theorem}[Theorem \ref{upper_bound_reg}]
Let $G$ be a bipartite graph. Then for all $s \geq 1$, we have
\begin{equation}\label{eq:1}
\reg(I(G)^s) \leq 2s + \cochord (G) -1.
\end{equation}
\end{theorem}
We also compare
certain properties and invariants, algebraic as well as combinatorial,
of $G$ and $G'$ for several subclasses of bipartite graphs.  We prove
that if $G$ is either unmixed bipartite (Theorem \ref{unmi_colon}) or
$P_k$-free bipartite (Theorem \ref{bipartite_pk}) or $nK_2$-free
(Corollary \ref{nK2-free}), then so is $G'$. We also prove that the
induced matching number of $G'$ is at most that of $G$ (Proposition \ref{even_con_ind}). As a
consequence, we obtain an upper bound for $\reg(I(G)^s)$ when $G$
is a bipartite graph (Corollary \ref{upperbound_ind}).
Comparison between the graphs $G$ and $G'$ yields yet another
positive result, namely, a partial answer to a question posed by Banerjee,
\cite[Question 6.2.2]{banerjee_thesis}, on classifying all graphs $G$
and edges $e_1 \cdots e_s$ such that $$ \reg(I(G)) \geq
\reg (I(G)^{s+1}:e_1 \cdots e_s) \text{ for all } s \geq 1.$$ 
We obtain some sufficient conditions for this inequality to be true,
(Proposition \ref{reg-g-g'}) and as a consequence we prove that this
inequality holds true for unmixed bipartite graph, chordal bipartite,
whiskered bipartite graph, bipartite $P_6$-free graphs and connected
bipartite graphs with regularity equal to three. 

We then move on to compute precise expressions for the regularity of
powers of edge ideals.  In \cite{selvi_ha}, the authors raised the
question, for which graphs $G$, $\reg(I(G)^s) = 2s + \ab(G) - 1$ for
$s\gg 0$.  We observe that for certain classes of bipartite graphs,
the induced matching number coincides with the co-chordal cover
number, for example, unmixed bipartite, chordal bipartite and
whiskered bipartite.  We then use the upper bound (\ref{eq:1}) for
such classes of graphs to get $\reg(I(G)^s) = 2s + \ab(G) - 1$ for all
$s \geq 1$, (Corollary \ref{thm:ag=cg}).  As an immediate consequence,
we derive one of the main results of \cite{selvi_ha}, that the above
equality holds for forests. We also derive the main result of Alilooee
and Banerjee, in \cite{banerjee1}, that equality holds true for
connected bipartite graphs $G$ with $\reg(I(G)) = 3$.

The classes of graphs discussed earlier  have the property $\ab(G) =
\cochord(G)$. For bipartite $P_6$-free graphs, it not known whether
this equality holds true. However we prove that for such graphs, for all $s
\geq 1$, $\reg(I(G)^s) = 2s + \ab(G) - 1$. It was shown by Jacques,
\cite{sean_thesis}, that if $n \equiv 2(mod~3)$, then $\reg(I(C_n)) =
\ab(C_n) + 2$.  And, Beyarslan, H\`a and Trung proved that
$\reg(I(C_n)^s) = 2s + \ab(C_n) - 1$ for all $s \geq 2$,
\cite{selvi_ha}. If $G$ is the disjoint union of $C_{n_1}, \ldots,
C_{n_m}$ and $k$ edges, for some $k \geq 1$, then we obtain a precise
expression for $\reg(I(G)^s)$, (Theorem \ref{upperbound}). 
We also construct,
for each $t \geq 1$, a graph $G_t$ such that $\reg(I(G_t)^s) - [2s +
  \ab(G_t) - 1] = t$, (Example \ref{ex_reg=cochord}).
\section{Preliminaries}
\par In this section, we set up the basic definitions and notation needed for the main results. Let $G$ be a finite simple graph with vertex set $V(G)$
and edge set $E(G)$. 
 A subgraph $H \subseteq G$  is called \textit{induced} if $\{u,v\}$ is an edge of $H$ 
if and only if $u$ and $v$ are vertices of $H$ and $\{u,v\}$ is an edge of $G$.
For a vertex $u$ in a graph $G$, let $N_G(u) = \{v \in V (G)|\{u, v\} \in E(G)\}$ be the set of
neighbors of $u$. 
The \textit{complement} of a graph $G$, denoted by $G^c$, is the graph on the same
vertex set in which $\{u,v\}$ is an edge of $G^c$ if and only if it is not an edge of $G$.
A subset
$X$ of $V(G)$ is called \textit{independent} if there is no edge $\{x,y\} \in E(G)$ 
for $x, y \in X$. The independence number $\alpha(G)$ is the maximum size of an independent set.
Let $C_k$ denote the cycle on $k$ vertices and $P_k$ denote the path on $k$ vertices.
The \textit{length} of a  path, or cycle is its number of edges.

A graph $G$ is called \textit{bipartite} if there are
two disjoint independent subsets $X, Y$ of $V(G)$ such that $X \cup Y
= V(G)$. 
 \par 
Let $G$ be a graph. We say $n$ non-adjacent edges
  $\{f_1,\ldots,f_n\}$ form an $nK_2$ in $G$ if $G$ does not have an edge with one 
endpoint in $f_i$ and the other in $f_j$ for all $i,j \in \{1,\ldots,n\}$ and $i \neq j$. A graph without $nK_2$ is called $nK_2$-free. If $n$ is $2$, then
$2K_2$-free graph also called \textit{gap-free} graph.
It is easy to see that, $G$ is gap-free if and only if $G^c$ contains no induced $C_4$.
Thus, $G$ is gap-free if and only if it does not contain two vertex-disjoint edges as
an induced subgraph.

\par A \textit{matching} in a graph $G$ is a subgraph consisting of pairwise disjoint edges. The largest size of a 
  matching in $G$ is called its matching number and denoted by $\ca(G)$ and
the minimum matching number of $G$, denoted by $\ba(G)$, is the minimum
cardinality of the maximal matchings of $G$. If the
subgraph is an induced subgraph, the matching is an \textit{induced matching}. The largest size of an induced matching in $G$ is called its induced 
matching number and denoted by $\ab(G)$. 

Let $G$ be a graph. A subset $C \subseteq V(G)$ is a vertex cover  of
$G$ if for each $e \in E(G)$, $e\cap C \neq \phi$.  If $C$ is minimal
with respect to inclusion, then $C$ is called \textit{minimal vertex
cover} of $G$.  A graph $G$ is called \textit{unmixed (also called
well-covered)} if all minimal vertex covers of $G$ have the same
number of elements. 

For a graph $G$ on $n$ vertices, let $W(G)$ be the \textit{whiskered
graph} on $2n$ vertices obtained by adding a pendent vertex (an edge
to a new vertex of degree $1$) to every vertex
of $G$.

A graph $G$ is \textit{weakly chordal} if every induced cycle in both
$G$ and $G^c$ has length at most $4$
and $G$ is \textit{chordal bipartite} if it is simultaneously weakly chordal  
and bipartite. Equivalently, a bipartite graph is chordal bipartite if
and only if it has no induced cycle on six or more vertices.

For any undefined terminology and further basic properties of graphs,
we refer the reader to \cite{west}.

\begin{example}
Let $G$ be the graph with vertices $V(G)=\{x_1,\ldots,x_6\}$ given
below. 

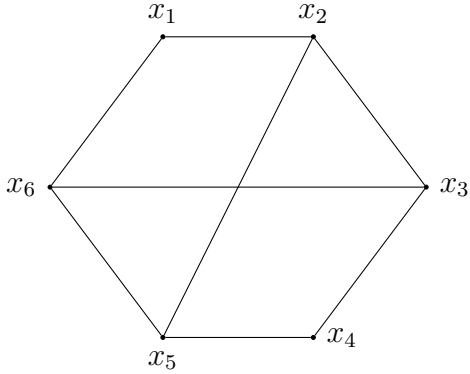
\begin{minipage}{\linewidth}
  \captionsetup[figure]{labelformat=empty}
 \begin{figure}[H]
  
 \[\begin{tikzpicture}
	\vertex [fill] (x_1) at (2,4) [label=above:$x_1$] {};  
	\vertex [fill] (x_2) at (4,4) [label=above:$x_2$] {};
	\vertex [fill] (x_3) at (5.5,2) [label=right:$x_3$] {};
	\vertex [fill] (x_4) at (4,0) [label=right:$x_4$] {};
	\vertex [fill] (x_5) at (2,0) [label=below:$x_5$] {};
	\vertex [fill] (x_6) at (0.5,2) [label=left:$x_6$] {};
	\path
		(x_1) edge (x_2)
		(x_2) edge (x_3)
		(x_3) edge (x_4)
		(x_4) edge (x_5)
		(x_5) edge (x_6)
		(x_6) edge (x_1)
		(x_3) edge (x_6)
		(x_2) edge (x_5)
		 ;  
	 \node[text width=9.5cm] at (11.5,2)
	{Then $\bigg\{ \{x_2,x_3\},\{x_5x_6\} \bigg \}$ forms a matching, but not an induced matching since the induced
subgraph with vertices $\{x_2,x_3,x_5,x_6\}$ contains edges $\{x_3,x_6\}$ and $\{x_2,x_5\}$. 
The set $\big\{\{x_1,x_2\},\{x_3,x_4\},\{x_5,x_6\}\big \}$ forms a matching of $G$ and  $\big\{ \{x_2,x_5\}, \{x_3,x_6\} \big \}$
also form a matching, 
the set $\{x_1,x_3,x_5\}$
forms an independent set of $G$. It is not hard to verify that $\ca(G)=3, \ba(G)=2,\ab(G)=1$ and $\alpha(G)=3$ .
It can also be noted that
$\{x_2,x_3,x_5, x_6\}$ and $\{x_2, x_4, x_6\}$ are minimal vertex
covers of $G$. Therefore $G$ is not unmixed.};
\end{tikzpicture}\]
\end{figure}
\end{minipage}
\end{example}

We recall the definition of even-connectedness and some of its important
properties from \cite{banerjee}.
\begin{definition}\label{even_connected} Let $G$ be a graph. Two vertices $u$ and $v$ ($u$ may be same as $v$) are said to be even-connected with respect to an $s$-fold products 
$e_1\cdots e_s$, where $e_i$'s are edges of $G$, not necessarily distinct, if there is a path $p_0p_1\cdots p_{2k+1}$, $k\geq 1$ in $G$ such that:
\begin{enumerate}
 \item $p_0=u,p_{2k+1}=v.$
 \item For all $0 \leq \ell \leq k-1,$ $p_{2\ell+1}p_{2\ell+2}=e_i$ for some $i$.
 \item For all $i$, $ \mid\{\ell \geq 0 \mid p_{2\ell+1}p_{2\ell+2}=e_i \}\mid
   ~ \leq  ~ \mid \{j \mid e_j=e_i\} \mid$.
 \item For all $0 \leq r \leq 2k$, $p_rp_{r+1}$ is an edge in $G$.
\end{enumerate}
\end{definition}
\begin{example}Let $I(G)=(x_1x_2,x_1x_5,x_2x_5,x_2x_3,x_3x_4,x_4x_5)\subset k[x_1,\ldots,x_5]$. Then $(I(G)^2:x_2x_5)=I(G)+(x_1^2,x_1x_3,x_1x_4)$. Note that, 
$x_1$ is even-connected to itself and $\{x_1,x_3\}$, $\{x_1,x_4\}$ are even-connected with respect to $x_2x_5$.
\end{example}
The following theorem, due to Banerjee, is used repeatedly throughout this paper.
\begin{theorem} \label{regularity_compa}\cite[Theorem 5.2]{banerjee} For any finite simple graph $G$ and any $s \geq 1$, let the set of minimal 
monomial generators of $I(G)^s$ be $\{m_1,\ldots,m_k\}$, then 
\begin{center}
$\reg(I(G)^{s+1}) \leq \max \{\reg(I(G)^{s+1}:m_\ell)+2s,~1 \leq \ell \leq k, ~\reg(I(G)^s)\}$. 
\end{center}
\end{theorem}
Next theorem describes all the minimal generating set of an ideal $(I(G)^{s+1}:M)$, where $M$ is minimal generator of $I(G)^s$ for $s\geq 1$.
\begin{theorem}\label{even_connec_equivalent}\cite[Theorem 6.1 and Theorem 6.7]{banerjee} Let $G$ be a graph with edge ideal
$I = I(G)$, and let $s \geq 1$ be an integer. Let $M$ be a minimal generator of $I^s$.
Then $(I^{s+1} : M)$ is minimally generated by monomials of degree 2, and $uv$ ($u$ and $v$ may
be the same) is a minimal generator of $(I^{s+1} : M )$ if and only if either $\{u, v\} \in E(G) $ or $u$ and $v$ are even-connected with respect to $M$.
 \end{theorem}
Further, Alilooee and Banerjee studied the even-connection in the
context of bipartite graphs and showed that they behave well under
even-connections.
 \begin{theorem}\label{even_con_bipartite}\cite[Proposition 3.5]{banerjee1} Let $G$ be a bipartite graph and $s \geq 1$ be an integer. Then for
every $s$-fold product $e_1\cdots e_s$, $(I(G)^{s+1}:e_1\cdots e_s)$ is a quadratic squarefree monomial
ideal. Moreover the graph $G'$ associated to $(I(G)^{s+1} : e_1 \cdots e_s)$ is bipartite on the
same vertex set and same bipartition as $G$.  
 \end{theorem}

Polarization is a process to obtain a squarefree monomial ideal from
a given monomial ideal. For details of polarization we refer to \cite{kummini_thesis}, \cite{ms2005}.
\begin{definition}
 Let $f=x_1^{m_1} \cdots x_n^{m_n}$ be a monomial in $R=k[x_1,\ldots,x_n]$.
 Let $\widetilde{R}=k[x_{11},x_{12},\ldots,x_{21},x_{22},\ldots,x_{n1},x_{n_2},\ldots]$. 
 Then a \textit{polarization} of $f$ in $\widetilde{R}$ is the squarefree monomial 
 $\widetilde{f}= x_{11} \cdots x_{1m_1}x_{21}\cdots x_{2m_2} \cdots x_{n1} \cdots x_{nm_n}$. 
 If $f_1,\cdots,f_m \in R$ are monomials and 
 $I = (f_1,\cdots,f_m)$, then we call the squarefree monomial ideal $\widetilde{I}$ generated by the polarization
 of the $f_i$'s in a larger polynomial ring $\widetilde{R}$,
 the polarization of $I$.
\end{definition}
Let $G$ be a graph and $I(G)$ denote the edge ideal of $G$. Then for
any $s \geq 1$ and edges $e_1, \ldots, e_s$ of $G$, $\widetilde{I} =
\widetilde{(I(G)^{s+1} : e_1\cdots e_s)}$ is a squarefree quadratic
monomial ideal, by Theorem \ref{even_connec_equivalent}. Hence there
exists a graph $G'$ associated to $\widetilde{I}$. Note also that $G$
is a subgraph of $G'$.
\begin{example}
Let $G = C_3$ and $I(G)=(x_1x_2,x_2x_3, x_1x_3) \subset
k[x_1,x_2,x_3]$. 
Then $I = (I(G)^2 : x_1x_3) = I(G) + x_2^2$. Therefore,
$\widetilde{I} \subset k[x_1,x_2,x_3,x_4]$ is given by $\widetilde{I} =
I(G) + (x_2x_4)$. Then $G'$ is given the graph $G$ with the edge
$\{x_2,x_4\}$ attached to $G$.
 \end{example}
Let $M$ be a graded $R = k[x_1,\ldots,x_n]$ module. For non-negative
integers $i, j$, let $\beta_{ij}(M)$ denote the $(i,j)$-th graded Betti
number of $M$.
 \begin{theorem}\cite[Proposition 1.3.4]{kummini_thesis} \cite[Exercise 3.15]{ms2005} \label{polarization}
 Let $I \subseteq R=k[x_1,\ldots,x_n]$ be a monomial ideal. 
 If $\widetilde{I} \subseteq \widetilde{R}$ is a polarization of $I$, then for 
 all $\ell,j$, $\beta_{\ell,j}(R/I)=\beta_{\ell,j}(\widetilde{R} / \widetilde {I})$. In particular $\reg(R/I)=\reg (\widetilde R/\widetilde I)$.
 \end{theorem}

\section{Upper bound for the regularity of powers of bipartite graph ideals}
In this section, we study the powers of the edge ideals of bipartite graphs. 
We obtain an upper bound for the regularity of powers of edge ideals
of bipartite graphs in terms of co-chordal cover number. 

\begin{definition}
A graph $G$ is chordal (also called triangulated) if every induced
cycle in $G$ has length $3$, and is co-chordal if the complement graph $G^c$ is chordal.
\par The co-chordal cover number, denoted $\cochord(G)$, is the
minimum number $n$ such that there exist co-chordal subgraphs
$H_1,\ldots, H_n$ of $G$ with $E(G) = \bigcup_{i=1}^n E(H_i)$. 
\end{definition}
In the following, we relate the co-chordal cover number of a graph
with that of its polarization.  As a consequence we obtain an upper
bound for $\reg(I(G)^s)$ for $s \geq 1$, when $G$ is a bipartite graph.
\begin{theorem}\label{cochord} 
Let $G$ be a graph and $e$ be an edge of $G$. Let $G'$ be the graph associated to
$\widetilde{(I(G)^{2}:e)}$. Then $$\cochord(G') \leq \cochord(G).$$
\end{theorem}

\begin{proof}
Let $\cochord(G) = n$. Then there exist co-chordal subgraphs $H_1,
\ldots, H_n$ such that $E(G) = \bigcup_{m=1}^n E(H_m)$.
If $G = G'$, then we are done. 
Let  $\{p_1,p_2\}=e$, $N_G(p_1) \setminus \{p_2\}=\{p_{1,1},\ldots,p_{1,s}\}$
and $N_G(p_2) \setminus \{p_1\}=\{p_{2,1},\ldots,p_{2,t}\}$.
For any two vertices $x,y$, set
$$\{[x,y]\} = \left\{
  \begin{array}{ll}
	\{x,y\} & \mbox{ if } x \neq y;\\
	\{x,z_x\} & \mbox{ if } x = y, \mbox{ where } z_x \mbox{ is a new
	vertex.} 
  \end{array}\right.$$
Note that for $x,y \in V(G)$, $x$ is even-connected to $y$ with
respect to $e$ in $G$ if
and only if $\{x,y\} = \{p_{1,i}, p_{2,j}\}$ for some $1 \leq i \leq s, 1 \leq j \leq t$. 
Therefore $$E(G')=E(G) \cup \{[p_{1,1},p_{2,1}]\}\cup \cdots \cup
\{[p_{1,1},p_{2,t}]\} \cup \cdots \cup \{[p_{1,s},p_{2,1}]\} \cdots
\{[p_{1,s},p_{2,t}]\}.$$ For each $1 \leq \mu \leq s$, $\{p_1,
p_{1,\mu}\} \in E(H_m)$ for some $1 \leq m \leq n$. We add certain
even-connected edges to $H_m$ with a rule as described below, to get 
a new graph $H_m'$:

Since $H_m$ is co-chordal, by \cite[Lemma 1 and Theorem 2]{benzaken},
there is an ordering of edges of $H_m$, $f_1 < \cdots < f_{t_m}$, such
that for $1\leq r \leq t_m, ~ (V(H_m), \{f_1,\ldots,f_r\})$ has no
induced subgraph isomorphic to $2K_2$. 
\vskip 2mm
  If for $1 \leq \mu \leq s, ~\{p_{1,\mu},p_1\} =
	f_k$ for some $1 \leq k \leq t_m$, then set $$\cdots < f_k < \{[p_{1,\mu},p_{2,1}]\}< \cdots
	<\{[p_{1,\mu},p_{2,t}]\}< f_{k+1} < \cdots .$$
Then we have $E(G') = \bigcup_{m=1}^n E(H_m')$.
We claim that $H_m'$ is co-chordal. 
Let $E(H_m')=\{g_1,\ldots,g_{t_{m_1}}\}$ be edge set of $H_m'$ and
linearly ordered as given above. By Lemma 1 and Theorem 2 of
\cite{benzaken}, it is enough to prove that for $1\leq r' \leq t_{m_1},
~ (V(H_m'), \{g_1,\ldots,g_{r'}\})$ has no induced subgraph isomorphic
to $2K_2$. Suppose
$H_m'$ is not co-chordal. Then there exists a least $i$ such that
$(V(H_m'), \{g_1,\ldots,g_i\})$ has an induced $2K_2$-subgraph, say
$\{g_j, g_i\}$. Since $H_m$ is co-chordal, $g_j$ and $g_i$ cannot be
in $E(H_m)$ simultaneously.  
\vskip 2mm \noindent
\textsc{Case 1:} Suppose $g_j \in E(H_m')\setminus E(H_m)$ and
$g_i=\{x_\alpha,x_\beta\} \in E(H_m)$.  Let
$g_j=\{[p_{1,k},p_{2,\ell}]\}$, for some $1 \leq k \leq s$ and $1 \leq \ell
\leq t$. By construction, we have
$$g_{j'}=\{p_{1,k},p_1\} < g_j <g_i.$$ Since $g_{j'}, g_i \in E(H_m)$,
they cannot form an induced $2K_2$-subgraph of $H_m$. Therefore,
either $g_{j'}$ and $g_i$ have a vertex in common or there exist an
edge $g_h \in E(H_m)$ such that $g_h < g_i$ connecting $g_{j'}$ and
$g_i$. If $g_{j'}$ and $g_i$ have a vertex in common, then 
this contradicts the assumption that $\{g_j, g_i\}$ form an
induced $2K_2$-subgraph. Suppose $g_h$ is a an edge
connecting $g_{j'}$ and $g_i$.
Let $g_h=\{p_1,x_\alpha\}$ with
$x_\alpha \neq p_2$. Then $x_\alpha \in N_{H_m}(p_1)$ and hence by
construction, there is a new edge $\{[x_\alpha,p_{2,\ell}]\} \in E(H_m')$
with the ordering $$g_h < \{[x_\alpha,p_{2,\ell}]\} < g_i.$$  This also
contradicts the assumption that $\{g_j, g_i\}$ is an induced
$2K_2$-subgraph. Now if $g_h=\{p_1,x_\alpha\}$ and $x_\alpha = p_2$, then $x_\beta \in N_{H_m}(p_2)$.
Therefore there is an edge $\{[p_{1,k},x_\beta]\} \in
E(H_m')$ with the ordering $$g_{j'} < \{[p_{1,k}, x_\beta]\} < g_i.$$  This also
contradicts the assumption that $\{g_j, g_i\}$ is an induced
$2K_2$-subgraph. Similarly, if $g_h = \{p_1, x_\beta\}, \{p_{1,k},
x_\alpha\}$ or $\{p_{1,k}, x_\beta\}$ for some $k$, then one arrives
at a contradiction.

If $g_j \in E(H_m)$ and $g_i \in E(H_m') \setminus E(H_m)$, then we
get contradiction in a similar manner.
\vskip 2mm \noindent
\textsc{Case 2:} Suppose $g_j, g_i \in E(H_m')\setminus E(H_m)$. Let $g_i=\{[p_{1,k'},p_{2,\ell'}]\}$
and $g_j=\{[p_{1,k},p_{2,\ell}]\}$, for some $1 \leq k,k' \leq s$, $1 \leq \ell,\ell' \leq t$. 
By construction, we have $$g_{j'}=\{p_{1,k},p_1\}<g_{j}<g_{i'}=\{p_{1,k'},p_{1}\}<g_{i}.$$
Since $p_{1,k} \in N_G(p_1)$ and $p_{2,\ell'} \in N_G(p_2)$,
by construction, there exists the edge $\{[p_{1,k}, p_{2,\ell'}]\}$ in
$H_m'$ with the ordering
$$g_{j'}<\{[p_{1,k},p_{2,\ell'}]\}<g_{i'}.$$
This contradicts the assumption that $\{g_j, g_i\}$ is an induced
$2K_2$-subgraph.

Therefore $H_m'$ is a co-chordal graph for $1 \leq m \leq n$ and
$E(G') = E(H_1') \cup \cdots \cup E(H_n').$ 
Hence $\cochord(G') \leq n$.
\end{proof}

In \cite[Corollary 3.6]{banerjee1}, Alilooee and Banerjee proved that, if the
minimal free resolution of $I(G)$ is linear, then so is the minimal
free resolution of $(I(G)^{s+1}:e_1 \cdots e_s)$. By applying this
result recursively, one can see that if the minimal free resolution of
$I(G)$ is linear, then so is the minimal free resolution of
$(I^2:e_{i_1})^2: \cdots )^2: e_{i_j})$. 
We reprove this result as a consequence of Theorem
\ref{cochord}.

\begin{corollary}\label{polarization_even}  
Let $G$ be any graph with edge ideal $I=I(G)$ and $e_1,\ldots, e_s$,
$s \geq 1$ be some edges of $G$ which are not necessarily distinct. If
the minimal free resolution of $I$ is linear, then so is the minimal
free resolution of 
$(((I^2:e_{i_1})^2: e_{i_2})^2 \cdots)^2 : e_{i_m}),$
where $\{i_1,\ldots,i_m\}\subseteq \{1,\ldots,s\}$.
\end{corollary}
\begin{proof}
Fr\"oberg proved that $I(G)$ has a linear minimal free resolution if
and only if $\cochord(G) = 1$, \cite[Theorem 1]{froberg}. Let $G_1'$ be the graph associated with the polarization of $(I(G)^2 :
e_{i_1})$. Therefore from Theorem \ref{cochord}, $\cochord(G'_1)=1$.
For $j \geq 2,$ define $G_j'$ to be the graph associated with
the polarization of $(I(G_{j-1}')^2 : e_{i_j})$. Now recursively
applying Theorem \ref{cochord} and \cite[Theorem 1]{froberg}, we get
the assertion.
 \end{proof}

The following corollary helps to obtain upper bound for the
asymptotic regularity of edge ideals of bipartite graphs.
\begin{corollary}\label{bipartite_chord_com} Let $G$ be a bipartite
graph and $e_1,\ldots, e_s$, $s \geq 1,$ be some edges of $G$ which are not
necessarily distinct. Let $G'$ be the graph associated to
$(I(G)^{s+1}:e_1 \cdots e_s)$. Then $$\cochord(G') \leq \cochord(G).$$
 \end{corollary}
\begin{proof} Since $G$ is a bipartite graph, it follows by Theorem
\ref{even_con_bipartite} that the graph $G'$ associated to
$\widetilde{((\widetilde{(I(G)^{2}:e_{1})}^2: \cdots )^2:e_{s})}$ is
bipartite on the same vertex set and same bipartition as on $G$.  By
\cite[Lemma 3.7]{banerjee1}, $(((I(G)^{2}:e_1)^2: \cdots
)^2:e_s)=(I(G)^{s+1}:e_1 \cdots e_s)$.  Therefore by applying Theorem
\ref{cochord} recursively, we get $\cochord(G') \leq \cochord(G)$.
\end{proof}
If $G$ is not a bipartite graph, then the equality
$(((I(G)^{2}:e_1)^2: \cdots)^2:e_s)=(I(G)^{s+1}:e_1 \cdots e_s)$ need
not necessarily hold, see the example below.
This example further shows that $(((I(G)^2:e_1))^2:e_2)$ has a linear
minimal free resolution need not necessarily imply that
$(I(G)^3:e_1e_2)$ has a linear minimal free resolution.

\begin{example} 
Let $I=(x_1x_7,x_1x_2,x_2x_3,x_2x_6,x_3x_4,x_3x_5, x_4x_5,x_6x_8)
\subset R=k[x_1,\ldots, x_8]$ and $G$ be the associated graph.  Let
$G_1$ and $G_2$ be the graphs associated to
$$\widetilde{(I^3:x_2x_3x_4x_5)}=I+(x_1x_3,x_1x_5,x_1x_4,x_3y_1,x_3x_6,x_4x_6,x_5x_6)\subset R_1=R[y_1]$$ and
$$\widetilde{((\widetilde{(I^2:x_2x_3)})^2:x_4x_5)}=\widetilde{(I^3:x_2x_3x_4x_5)}+(x_6y_3,x_1x_6,x_1y_2)\subset 
R_1[y_2,y_3]$$ respectively.
Then it can easily be seen that $G_1^c$ is not chordal and $G_2^c$ is chordal. 
Therefore by \cite[Theorem 1]{froberg}, $I(G_1)$ does not have linear minimal free resolution
and $I(G_2)$ has linear minimal free resolution. By Theorem \ref{polarization},
$(((I^2:x_2x_3))^2:x_4x_5)$ has linear minimal free resolution but $(I^3:x_2x_3x_4x_5)$ does not have
linear minimal free resolution.
 \end{example}

We now prove an upper bound for $\reg(I(G)^s)$, when $G$ is a bipartite graph.
\begin{theorem}\label{upper_bound_reg} 
Let $G$ be a bipartite graph. Then for all $s \geq 1$, 
 $$\reg(I(G)^s) \leq 2s + \cochord (G) -1.$$
\end{theorem}
\begin{proof} We prove by induction on $s$. If $s=1,$ then the
assertion follows from \cite[Theorem 1]{russ}. 
Assume $s >1$. By applying Theorem \ref{regularity_compa} and using
induction, it is enough to prove that for edges $e_1, \ldots, e_s$ of
$G$ (not necessarily distinct),
$\reg(I(G)^{s+1}:e_1 \cdots e_s) \leq \cochord(G)+1$ for all $s > 1$. 
Let $G'$ be the graph associated to the ideal $(I(G)^{s+1}:e_1 \cdots e_s)$.  
\begin{eqnarray*}
 \reg((I(G)^{s+1}:e_1 \cdots e_s))& \leq & \cochord(G')+1, \\
 &\leq & \cochord(G)+1, 
\end{eqnarray*}
where the first inequality follows from \cite[Theorem 1]{russ} and
the second inequality follows from Corollary \ref{bipartite_chord_com}.
Hence $\reg(I(G)^s) \leq 2s+\cochord(G)-1$ for all $s\geq 1$.
\end{proof}
 \begin{remark}
The inequality given in Theorem \ref{upper_bound_reg} could be asymptotically strict.
For example, if $G=C_8$, then one can see that the co-chordal
subgraphs of $C_8$ are paths with at most $3$ edges so that
$\cochord(G)=3$. On the other hand, by  \cite[Theorem 5.2]{selvi_ha},
$\reg(I(G)^s)=2s+1 < 2s+2$ for all $s \geq 2$.  
 \end{remark}

From \cite[Theorem 4.5]{selvi_ha} and Theorem \ref{upper_bound_reg} for any bipartite graph $G$, we have
\begin{equation}\label{eq:3}
2s+\ab(G)-1 \leq \reg(I(G)^s) \leq 2s+ \cochord (G)-1 \text{ for any } s \geq 1.
\end{equation}

As a consequence of $(\ref{eq:3})$, we derive the following result of
Alilooee and Banerjee: 
\begin{corollary}\cite[Proposition 2.15]{banerjee1} Let $G$ be a bipartite graph. Then following are equivalent
\begin{enumerate}
 \item $I(G)$ has a linear presentation.
 \item $I(G)^s$ has a linear resolution for all $s \geq 1$.
 \item $G^c$ is chordal.
\end{enumerate}
\end{corollary}
\begin{proof} 
It is known that for a graph $G$, if $G$ is bipartite, then $\ab(G) = 1$ if and only if $\cochord(G) = 1$.
Therefore the equivalence of the three statements follow directly 
from $(\ref{eq:3})$ and \cite[Proposition 1.3]{nevo_peeva}. 
\end{proof}

For any graph $G$,
H{\`a} and Van Tuyl proved that
$\reg(I(G)) \leq \ca(G)+1$, \cite[Theorem 6.7]{ha_adam}. Woodroofe
then proved a strong 
generalization of their result, namely, $\reg(I(G))\leq \ba(G)+1$ for
any graph $G$, \cite[Theorem 2]{russ}.
Let $G$ be a graph and
$\{z_1,\ldots,z_t\}$ be a minimum maximal matching of $G$. Let $Z_i$
be the subgraph of $G$ with $E(Z_i)=z_i \cup \{ \text{ adjacent edges
of $z_i$}\}$. Then for each $i$,  $Z_i$ is a co-chordal subgraph of $G$ 
and $E(G) = \cup_{i=1}^t E(Z_i).$ 
Hence $\cochord(G)\leq \ba(G)$.
Therefore, for any bipartite graph $G$, it follows from \cite[Theorem 4.6]{selvi_ha} and Theorem
\ref{upper_bound_reg} that
\begin{equation}\label{eq:4}
 2s+\ab(G)-1 \leq \reg(I(G)^s)\leq 2s+\ba(G)-1.
\end{equation}
A \textit{dominating induced matching} of $G$ is an induced matching which also
forms a maximal matching of $G$. 
If $G$ has a dominating induced
matching, then $\ab(G)=\ba(G)$. Hence for any bipartite graph $G$
with dominating induced matching, we have, $$\reg(I(G)^s)=2s+\ab(G)-1 \text{ for $s \geq 1$.}$$
In \cite{hibi}, Hibi et al. 
characterized graphs with dominating induced matchings 
and also $G$ satisfying $\ab(G)=\ba(G)$.

\section{Relation between $G$ and $G'$}
Let $G$ be a graph and $e_1, \ldots, e_s$ be edges of $G$. Let $G'$ be
the graph associated with $\widetilde{(I(G)^{s+1} : e_1\cdots e_s)}$. In
this section, we compare certain algebraic and combinatorial
properties of $G$ and $G'$. 
Using these comparisons we obtain upper bounds for the
regularity of powers of edge ideals of bipartite graphs.


We begin by considering unmixed bipartite graphs. 

\begin{theorem}\label{unmi_colon} 
If $G$ is an unmixed bipartite graph, then so is the graph $G'$
associated to $(I(G)^{s+1}:e_1 \cdots e_s)$, for any $s$-fold product
$e_1 \cdots e_s$  and $s \geq 1$.\footnote{In a personal
communication, we have been informed that Banerjee and Mukundan have
also obtained Theorem \ref{unmi_colon}.}
\end{theorem}
\begin{proof}
We prove the result using induction on s. 
Let $G$ be an unmixed 
bipartite graph. Then by \cite[Theorem 1.1]{Villarreal}, there exists
a partition $V_1=\{x_1,\ldots,x_n\}$ and
$V_2=\{y_1,\ldots,y_n\}$ with $V(G) = V_1 \cup V_2$. 
First we show that the graph $G'$ associated to $(I(G)^2 : e)$ is an unmixed 
bipartite graph for any edge $e$ in $G$. By Theorem \ref{even_con_bipartite}, 
$(I(G)^{2} : e)$ is bipartite on the same vertex set having the
same bipartition as in $G$.
Since $\{x_i,y_i\}\in E(G')$ for all $i$,  by \cite[Theorem 1.1]{Villarreal} 
we need to show that $\{x_i,y_k\} \in E(G')$, 
if $\{x_i,y_j\},\{x_j,y_k\} \in E(G')$ for distinct $i,j,k$.
 \vskip 2mm
\noindent
\textsc{Case I:} Suppose $\{x_i,y_j\},\{x_j,y_k\} \in E(G)$. 
Since $G$ is an unmixed bipartite graph, there is an edge $\{x_i,y_k\}
\in E(G)$, by \cite[Theorem 1.1]{Villarreal}. Hence $\{x_i,y_k\} \in E(G')$.
 \vskip 2mm
\noindent
\textsc{Case II:}  
Suppose $ \{x_i,y_j\} \in E(G)$ and $\{x_j,y_k\} \notin E(G)$.  Let
$x_jp_1p_2y_k$ be an even-connection between $x_j$ and $y_k$ with
respect to $e=p_1p_2$. Since $\{x_i,y_j\},\{x_j,p_1\} \in E(G)$, by
\cite[Theorem 1.1]{Villarreal}, there is an edge $\{x_i,p_1\} \in
E(G)$.  Therefore there is an even-connection $x_ip_1p_2y_k$ with
respect to $e$. Hence $\{x_i,y_k\} \in E(G')$.
\vskip 2mm
\noindent
\textsc{Case III:}        
If $ \{x_i,y_j\} \notin E(G)$ and $\{x_j,y_k\} \in E(G)$.  Let
$x_ip_1p_2y_j$ be an even-connection between $x_i$  and $y_j$ with
respect to $e=p_1p_2$.  Since $\{p_2,y_j\},\{x_j,y_k\} \in E(G)$, by
\cite[Theorem 1.1]{Villarreal}, there is an edge $\{p_2,y_k\} \in
E(G)$.  Therefore there is an even-connection $x_ip_1p_2y_k$, with
respect to $e$. Hence $\{x_i,y_k\} \in E(G')$.
\vskip 2mm
\noindent
\textsc{Case IV:} 
If $\{x_i,y_j\},\{x_j,y_k\} \notin E(G)$.  Consider the
even-connections $x_ip_1p_2y_j$ and  $x_jp_1p_2y_k$ between  $x_i$,
$y_j$ and $x_j$, $y_k$ respectively with respect to $e$.  Then there
is an even-connection $x_ip_1p_2y_k$ between $x_i$ and $y_k$ with
respect to $e$. Therefore $\{x_i,y_k\} \in E(G')$.  Hence $G'$ is an
unmixed bipartite graph.

\vskip 2mm \noindent
Assume by induction that for any unmixed bipartite graph $G$, 
$(I(G)^{s}:e_1 \cdots e_{s-1})$ is an unmixed bipartite graph
for any $(s-1)$-fold product and $s > 1$. 
By \cite[Lemma 3.7]{banerjee1}, we have $(I(G)^{s+1}:e_1 \cdots e_s)=
\big( (I(G)^2:e_i)^s: \prod_{j \neq i}e_j  \big )$. By the case $s=1$,
the graph associated to $(I(G)^2:e_i)$ is an unmixed bipartite graph,
say $G_i$. Therefore by induction the graph associated to
$(I(G_i)^s : \prod_{j\neq i}e_j)$ is an unmixed bipartite graph. This
completes the proof of the theorem.

\end{proof}
The following example shows that Theorem
\ref{unmi_colon} is not true if the graph is not bipartite.
\begin{example}\label{unmixed}
Let $I=(x_1x_4,x_1x_2,x_1x_3,x_2x_3,x_2x_5,x_3x_6)\subset
R=k[x_1,\ldots,x_6]$ and $G$ be the graph associated to $I$. Then $G$
is unmixed, but not bipartite. Taking $e_1 = \{x_1,x_2\}$, we get 
$\widetilde{(I^2:e_1)}=I+(x_4x_5,x_3x_5,x_3x_4,x_3y_1) \subset
R[y_1]$. Let $G'$ be the graph associated with
$\widetilde{(I^2:e_1)}$. Then it
can be seen that, $(x_1,x_3,x_5)$ and $(x_1,x_2,x_4,x_5,x_6,y_1)$ are
minimal vertex covers of $G'$. Therefore, $G'$ is not unmixed. 
\end{example}

A graph $G$ is called $H$-free for some graph $H$ if $G$ does not
contain an induced subgraph isomorphic to $H$. 
B{\i}y{\i}ko{\u{g}}lu and
Civan proved that if $G$ is bipartite $P_6$-free, then $\reg(I(G)) =
\ab(G) + 1$, \cite[Theorem 3.15]{biyikoglu_civan}. Below we prove that
if $G$ is $P_6$-free, then so is $G'$. 
This result is crucial in obtaining a precise expression, in the next
section, for the regularity of $I(G)^s$.

\begin{theorem}\label{bipartite_pk} 
If $G$ is a bipartite $P_k$-free graph for some $k \geq 4$, then so is
the graph $G'$ associated to $(I(G)^{s+1}:e_1 \cdots
e_s)$, for any $s$-fold product $e_1 \cdots e_s$ and $s \geq 1.$
\end{theorem}
\begin{proof} 
By \cite[Lemma 3.7]{banerjee1}, it is enough to prove the assertion for $s = 1$. 
Let $G$ be a
bipartite $P_k$-free graph, for some $k \geq 4$. First we show that
the graph $G'$ associated to $(I(G)^2 : e)$ is $P_k$-free
for any edge $e$ in $G$. Note that $G'$ is bipartite on the same
vertex set.
Suppose $G'$ has an induced path $P_k :
x_1x_2\cdots x_k$. First assume that $E(P_k) \setminus E(G)$ has only
one edge, say $\{x_i, x_{i+1}\}$. Let the  
even-connection be $x_ip_1p_2x_{i+1},$ where $e=\{p_1,p_2\}$. Note that
$\{x_i,x_{i+1}\} \cap \{p_1,p_2\} = \emptyset$. We first
show that the vertices $p_1$ and $p_2$ cannot be equal to or adjacent
to $x_j$ for $j < i-1$ and $j > i+2$.

\vskip 2mm \noindent
\textsc{Claim I:} $p_1 \neq x_j$ for $j \neq i-1$ and $p_2 \neq x_j$
for $j \neq i+2$.\\
If $p_1 = x_j$ for some $j \neq i-1,$ then $\{x_j,x_i\} \in
E(G)$. This contradicts the assumption that $P_k$ is an induced path.
Similarly if $p_2 = x_j$ for some $j \neq i+2$, then $\{x_j, x_{i+1}\}
\in E(G)$, which is again a contradiction.

\vskip 2mm \noindent
\textsc{Claim II:} $\{p_1,x_j\} \notin E(G)$ for any $j \neq i,i+2$ and
$\{p_2,x_j\} \notin E(G)$ for any $j \neq i-1,i+1$.\\
Suppose $\{p_1, x_j\} \in E(G)$ for some $j \neq i,i+2$. Then
$x_jp_1p_2x_{i+1}$ is an even-connection between $x_j$ and $x_{i+1}$
so that $\{x_j, x_{i+1}\} \in E(G')$.
Since $j \neq i, i+2$, this is a contradiction to the assumption that
$P_k$ is an induced path. Similarly, it can be seen that $\{p_2, x_j\}
\notin E(G)$ for $j \neq i-1,i+1$. 

\vskip 2mm \noindent
Now we deal with the remaining possibilities. If $p_1 = x_{i-1},$ then
$p_2 \neq x_{i+2}$ and hence we have a path $P':
x_{i-1}p_2x_{i+1}x_{i+2}$ in $G$. 
Similarly, if $p_2 = x_{i+2}$,
then $p_1 \neq x_{i-1}$ and hence we get a path $P':
x_{i-1}x_ip_1x_{i+2}$. 
Now suppose $\{p_1,x_{i+2}\} \in E(G)$. 
If $\{p_2,x_{i-1}\} \in E(G)$,
then there is an even-connection $x_{i-1}p_2p_1x_{i+2}$ which is a
contradiction. Therefore $\{p_2,x_{i-1}\} \notin E(G)$. Hence we have a
path $P' : x_{i-1}x_ip_1x_{i+2}$ in $G$. If $\{p_2,x_{i-1}\} \in
E(G)$, then $\{p_1,x_{i+2}\} \notin E(G)$ and hence we have a path
$P': x_{i-1}p_2x_{i+1}x_{i+2}$ in $G$.
\vskip 2mm \noindent
Since $p_1$ and $p_2$ cannot be equal to $x_j$ for $j < i-1$ and $j >
i+2$, replace the segment $x_{i-1}x_ix_{i+1}x_{i+2}$ in $P_k$ with
$P',$ to obtain an induced path $x_1\cdots x_{i-2}P'x_{i+3}\cdots x_k$
of length $k-1$ in $G$ which contradicts our hypothesis that $G$ is
$P_k$-free.
\vskip 2mm \noindent
Now suppose $E(P_k) \setminus E(G)$ has more than one edge. By the
proof of \textsc{Claim I} and \textsc{Claim II}, there cannot be more
than two pairs of vertices which are even-connected. Moreover, if there
are two even-connections, then the evenly connected vertices have to
be $\{x_{i-1},x_i\}$ and $\{x_i,x_{i+1}\}$ for some $i$. Let the 
even-connections be, $x_{i-1}p_1p_2x_i$ and $x_ip_2p_1x_{i+1}$. Note
that, in this case, for $r = 1,2$, $p_r \neq x_j$ for any $j$ and
$\{p_r,x_j\} \notin E(G)$ for $j \neq i-1,i,i+1$. Therefore we have
path $P':x_{i-1}p_1x_{i+1}x_{i+2}$ in $G$. Since $p_1$ cannot be equal to $x_j$ 
for $j < i-1$ and $j >
i+2$, replace the segment $x_{i-1}x_ix_{i+1}x_{i+2}$ in $P_k$ with
$P'$ to obtain an induced path  $x_1\cdots x_{i-2}P'x_{i+3}\cdots x_k$ with $k$ vertices in $G$, which contradicts our hypothesis that $G$ is
$P_k$-free. Hence $G'$ is $P_k$-free graph. 
\end{proof}

The following result compares the induced matching numbers of $G$ and
$G'$.
\begin{proposition}\label{even_con_ind} 
Let $G$ be a graph and $G'$ be the graph associated to
$\widetilde{(I(G)^{s+1} : e_1\cdots e_s)}$ for $e_1, \ldots,
e_s \in E(G)$. Then $\ab(G') \leq \ab(G)$.
\end{proposition}
\begin{proof}  
Let $\{f_1,\ldots,f_q,f_{q+1},\ldots,f_{r},f_{r+1},\ldots,f_t\}$ 
be an induced matching of $G'$, where
\begin{enumerate}
 \item for $\ell=1,\ldots,q, ~ f_\ell \in E(G);$ 
 \item for $\ell=q+1,\ldots,r$, $f_\ell = \{u_\ell,v_\ell\}$ and $u_\ell \neq v_\ell$
   are vertices of $G$ even-connected with respect to $e_1 \cdots e_s$. 
 \item for $\ell=r+1,\ldots,t$, $f_\ell = \{u_\ell,u_\ell'\}$, and $u_\ell$ is
   even-connected to itself with respect to $e_1 \cdots e_s$ and
   $u'_\ell$ is new vertex.
\end{enumerate}
Let $u_\ell=p_0^\ell p_1^\ell \cdots p_{2k_\ell+1}^\ell=v_\ell$, for $\ell=q+1,\ldots,t,$ be an even-connection
between $u_\ell$ and $v_\ell$ ($u_\ell$ may be equal to $v_\ell$) with respect to
$e_1\cdots e_s$.
For $i=q+1,\ldots,t$, let $f_i' = \{p_0^i,p_1^i\}$. 

\vskip 2mm
\noindent
\textbf{Claim:} $\{f_1, \ldots, f_q, f_{q+1}', \ldots, f_t'\}$ is an
induced matching for $G$.
\vskip 2mm
\noindent
\textit{Proof of the claim:} Suppose this is not an induced matching.
Then, either there is a common vertex for two of the edges or there
exists an edge in $G$ connecting two of the edges in the above set. Since
$\{f_1, \ldots, f_t\}$ is an induced matching, we can see that $f_i$
and $f_j'$ cannot have a
common vertex, for any $1 \leq i \leq q, ~q+1 \leq j \leq t$.
Suppose $f_\ell'=\{p_0^\ell,p_1^\ell\}$ and $f_m'=\{p_0^m,p_1^m\}$ have
a common vertex. If $p_0^\ell=p_0^m$ or $p_0^\ell=p_1^m$ or $p_1^\ell=p_0^m$,
then this contradicts the assumption that $\{f_\ell,f_m\}$
form an induced matching. If $p_1^\ell=p_1^m$, then
there is an even-connection $$p_0^m(p_1^m=p_1^\ell)p_2^\ell \cdots p_{2k_\ell+1}^\ell$$
between $p_0^m$ and $p_{2k_\ell+1}$ with respect to $e_1 \cdots e_s$, which
contradicts the assumption that $\{f_\ell,f_m\}$
form an induced matching. Therefore $f_\ell'$ and $f_m'$ cannot have a common vertex.
Also, since
$\{f_1, \ldots, f_q\}$ is an induced matching in $G$, there cannot be
an edge connecting $f_i$ and $f_j$. Therefore the two possibilities
are,
\begin{enumerate}
  \item there exists an edge connecting $f_i$ and $f_j'$;
  \item there exists an edge connecting $f_\ell'$ and $f_m'$.
\end{enumerate}
\noindent
\textsc{Case 1:} Let $f_i$ and $f_j'= \{p_0^{j},p_1^{j}\},$ for some
$1 \leq i \leq q$ and $q+1 \leq j \leq t$, be connected by an edge, say 
$\{x_\alpha, x_\beta\}$ in $G$.
If either $x_{\alpha}=p_0^{j}$ or $x_{\beta}=p_0^{j}$, then this is
a contradiction to the assumption that $f_i$ and
$f_j$ form an induced matching in $G'$. 
Therefore either $x_{\alpha}=p_1^{j}$ or $x_{\beta}=p_1^{j}$. Suppose
$x_\alpha = p_1^j$. Then there is an even-connection 
$x_{\beta}p_1^j \cdots p_{2k_j+1}^{j}$ 
in $G$. 
This is also in contradiction to
the assumption that $f_i$ and $f_j$ is an induced
matching in $G'$. Therefore $x_\alpha \neq p_1^j$. Similarly one can
prove that $x_\beta \neq p_1^j$. Therefore, there cannot be any common
edge $\{x_\alpha, x_\beta\}$ between $f_i$ and $f_j'$ for any $j =
q+1,\ldots,t$.

  \vskip 2mm
\noindent
\textsc{Case 2:}  Suppose there exists an edge, say $\{x_\alpha,
x_\beta\}$, between
$f_\ell'=\{p_0^{\ell},p_1^{\ell}\}$ and $f_m' = \{p_0^{m},p_1^{m}\}$.
\begin{enumerate}
  \item If $\{x_{\alpha},x_{\beta}\}=\{p_0^\ell,p_0^m\}$, then
	$\{x_{\alpha},x_{\beta}\}$ is a common edge of
	$f_\ell$ and $f_m$, which is a
	contradiction to the assumption that $\{f_\ell, f_m\}$ is an induced
	matching in $G'$.

  \item If $\{x_{\alpha},x_{\beta}\}=\{p_0^\ell,p_1^m\}$, then there is
	an even-connection $$p_0^\ell p_1^mp_2^m \cdots p_{2k_m+1}^m$$ between
	$p_0^\ell$ and $p_{2k_m+1}^m$ with respect to $e_1 \cdots e_s$, which is again a
	contradiction.  Similarly, one arrives at a contradiction if
	$\{x_{\alpha},x_{\beta}\} = \{p_1^\ell,p_0^m\}$. 
	
  \item Suppose $\{x_{\alpha},x_{\beta}\}=\{p_1^\ell,p_1^m\}$. 
	If $\{p^\ell_{2\mu+1},p^\ell_{2\mu+2}\}$ and $\{p^m_{2\mu_1+1},p^m_{2\mu_1+2}\}$ have a common vertex, for
	some $0 \leq \mu \leq k_\ell-1$, $0 \leq \mu_1 \leq k_m-1$, then by \cite[Lemma 6.13]{banerjee}, $p_0^\ell$
	is even-connected either to $p_0^m$ or to $p_{2k_m+1}^m$. 
	Both contradicts the assumption that $f_\ell, f_m$ is an
	induced matching in $G'$.
	Therefore $\{p^\ell_{2\mu+1},p^\ell_{2\mu+2}\}$ and $\{p^m_{2\mu_1+1},p^m_{2\mu_1+2}\}$ does not have a common vertex
	for any $0 \leq \mu \leq k_\ell-1$, $0 \leq \mu_1 \leq k_m-1$. Then there is
	an even-connection $$p_{2k_\ell+1}^\ell p_{2k_\ell}^\ell \cdots p_1^\ell p_1^m
	\cdots p_{2k_m+1}^m$$ between $p_{2k_\ell+1}^\ell$ and $p_{2k_m+1}^m$, 
	which is also a contradiction to the assumption that $\{f_\ell,f_m\}$ is an induced matching. 
\end{enumerate}
Therefore $\{f_1,\ldots,f_q, f_{q+1}', \ldots, f_t'\}$
  form an induced matching of $G$. Hence $\ab(G') \leq \ab(G)$.
\end{proof}
As a consequence of Proposition \ref{even_con_ind}, we obtain an
upper bound for $\reg(I(G)^s)$ 
when $G$ is a bipartite graph. 
\begin{corollary}\label{upperbound_ind} 
Let $G$ be a bipartite graph with partitions $X$ and $Y$.
Then for all $s \geq 1$, 
$$\reg(I(G)^s) \leq 2s + \frac{1}{2}\left(\ab(G) + \min\{|X|,
|Y|\}\right)-1.$$
\end{corollary}
\begin{proof}
If $s = 1$, then this is proved in \cite[Theorem
4.16]{projective_civam}.  Let $G'$ be the graph associated to
$(I(G)^{s+1} : e_1\cdots e_s)$ for some edges $e_1, \ldots, e_s$ in
$G$. Since $\ab(G') \leq \ab(G)$ and $G'$ is also bipartite with
partitions $X$ and $Y$, the assertion now follows from Theorem
\ref{regularity_compa} and induction. 
 \end{proof}
 
By \cite[Theorem 4.5]{selvi_ha} and Corollary \ref{upperbound_ind}, for any bipartite graph $G$, we have
\begin{equation}\label{eq:31}
2s+\ab(G)-1 \leq \reg(I(G)^s) \leq 2s+ \frac{1}{2}\left(\ab(G) + \min\{|X|,
|Y|\}\right)-1 \text{ for any } s \geq 1.
\end{equation}
Let $G$ be a bipartite graph with partitions $X$ and $Y$, say $|X|
\leq |Y|$. Then
by Corollary \ref{upperbound_ind}, $$\reg(I(G)^s) \leq 2s + |X|-1 \text{ for any } s \geq 1.$$

\begin{remark}
It is easy to see that for a bipartite graph $G$, $\ab(G) \leq
\min\{|X|, |Y|\}$. At the same time, the difference between 
$\ab(G)$ and $\min\{|X|, |Y|\}$ could be arbitrarily large, for
example in the case of complete bipartite graphs. If $\ab(G) =
\min\{|X|,|Y|\}$ or $\ab(G) = \min\{|X|,|Y|\} - 1$, then from
(\ref{eq:31}), it follows that $\reg(I(G)^s) = 2s + \ab(G) - 1$.
For example, let $H$ be a bipartite graph with $V(H) = X \cup Y$. Let
$X = \{x_1,\ldots, x_n\}$ and $n \leq |Y|$. Let $G$ be the graph
obtained by attaching at least one pendant vertex to each $x_i$'s in $H$.
Then $\ab(G) = n$ and hence $\reg(I(G)^s) = 2s + n - 1$.
\end{remark}

Another immediate consequence of the comparison between the induced
matching numbers is the relation between the
$nK_2$-free property of $G$ and $G'$.
\begin{corollary}\label{nK2-free}
If $G$ is an $nK_2$-free graph for some $n \geq 1$, then for any
$s$-fold product $e_1\cdots e_s$, the graph $G'$ associated to
$(\widetilde{I(G)^{s+1}:e_1\cdots e_s})$ is $nK_2$-free.  
\end{corollary}
\begin{proof}
If $G$ is $nK_2$-free, then $\ab(G) < n$. By Proposition \ref{even_con_ind},
$\ab(G') \leq \ab(G)$. Hence $\ab(G') < n$ and hence $G'$ is $nK_2$-free.
\end{proof}
Taking $n=2$ in the previous corollary, we obtain \cite[Lemma
6.14]{banerjee}. Using this result and 
\cite[Proposition 1.3]{nevo_peeva} we get:
\begin{corollary}\label{gapfree implies even_gapfree} 
If $G$ is gap-free, then so is the graph $G'$ associated to $\widetilde{(I(G)^{s+1}:e_1 \cdots e_s)}$, for every $s$-fold
product $e_1 \cdots e_s$. In other words, if $I(G)$ has linear
presentation, then so has $(I(G)^{s+1}:e_1 \cdots e_s)$, for every $s$-fold
product $e_1 \cdots e_s$.
 \end{corollary}
%
In \cite{banerjee_thesis}, Banerjee posed questions on the relation
between $I(G)$ and $(I(G)^{s+1} : e_1\cdots e_s)$ for some edges $e_1,
\ldots, e_s$ in $G$. In particular, he asked
\begin{question}\cite[Question 6.2.2]{banerjee_thesis}
  Classify $G$ and $e_1,\ldots,e_s$ such that $\reg(I(G)) \geq
  \reg(I(G)^{s+1}:e_1\cdots e_s)$.
\end{question}
As an application of our results Corollary \ref{bipartite_chord_com}
and Proposition \ref{even_con_ind}, we obtain some sufficient
conditions for which the above inequality holds true.
\begin{proposition}\label{reg-g-g'}
Let $G$ be any graph and $e_1,\ldots, e_s$ be edges of $G$, for some
$s \geq 1$. Let $G'$ be the graph associated to $\widetilde{(I(G)^{s+1}
: e_1\cdots e_s)}$.
Then the inequality
 $\reg(I(G)) \geq \reg(I(G)^{s+1}: e_1\cdots e_s)$ holds true if
\begin{enumerate}
  \item $\reg(I(G)^{s+1}:e_1 \cdots
	e_s)=\ab(G')+1$; ~ ~ or
  \item $G$ is bipartite with $\reg(I(G))=\cochord(G)+1$.
\end{enumerate}
\end{proposition}
\begin{proof}
 If $\reg(I(G)^{s+1}:e_1 \cdots e_s)=\ab(G')+1$, then
$$
\begin{array}{lcll}
    \reg(I(G)^{s+1}:e_1 \cdots e_s)&=&\ab(G')+1 &  \\
	& \leq & \ab(G)+1 & \mbox{ (by Proposition }\ref{even_con_ind})\\
    &=&\reg(I(G)). & \mbox{ (by (\ref{ag_reg_chg})) }
\end{array}
$$
If $G$ is bipartite and $\reg(I(G))=\cochord(G)+1$, then
$$
\begin{array}{lcll}
    \reg(I(G)^{s+1}:e_1 \cdots e_s)& \leq & \cochord(G')+1 & \mbox{ (by (\ref{ag_reg_chg})}) \\
	& \leq & \cochord(G)+1 & \mbox{ (by Corollary } \ref{bipartite_chord_com})\\
        &=& \reg(I(G)). 
\end{array}
$$
\end{proof}
The \textit{chromatic number} denoted by $\chi(G)$ is the smallest
number of colors possible in a proper vertex coloring of $G$, see \cite[Definition 5.1.4]{west}.
Note that, if $G$ is bipartite graph, then $\alpha(G)=\chi(G^c)$, see \cite[Chapter 8]{west}.
We now recall some results from the literature:

\begin{remark}\label{ag=cg}
Let $G$ be a graph. 
\begin{enumerate}
  \item If $G$ is unmixed bipartite, then Woodroofe proved that
	$\ab(G) = \cochord(G)$, \cite[Theorem 16]{russ}.

  \item If $G$ is a  graph, then Woodroofe proved that
	$\cochord(W(G)) = \chi(G^c)$ and $\alpha(G) = \ab(W(G))$, \cite[Lemma
	21]{russ}.

  \item If $G$ is a weakly chordal graph, then
	Busch-Dragan-Sritharan proved that $\ab(G) = \cochord(G)$,
	\cite[Proposition 3]{busch_dragan_sritharan}.
\end{enumerate}
\end{remark}
As an immediate consequence, we have 
\begin{corollary}\label{banerjee-qn}
Let $G$ be a bipartite graph and $e_1, \ldots, e_s$ be edges of $G$.
Then $\reg(I(G)^{s+1} : e_1\cdots e_s) \leq \reg(I(G))$ if 
\begin{enumerate}
  \item $G$ is unmixed;
  \item $G$ is weakly chordal;
  \item $G = W(H)$ for some bipartite graph $H$  ~ or
  \item $G$ is $P_6$-free;
\end{enumerate}
\end{corollary}
\begin{proof}
If $G$ is unmixed, weakly chordal or $G = W(H)$ for some bipartite
graph $H$, then the assertion follows from Remark \ref{ag=cg} and
Proposition \ref{reg-g-g'}. If $G$ is $P_6$-free, then by Theorem
\ref{bipartite_pk} $G'$ is also $P_6$-free. By  \cite[Theorem
3.15]{biyikoglu_civan}, $\reg(I(G')) = \ab(G') + 1$. Now the
assertion follows from Proposition \ref{reg-g-g'}.
\end{proof}

Below, we present an example to show that the inequalities
in Corollary \ref{bipartite_chord_com}, Proposition \ref{even_con_ind}
and Proposition \ref{reg-g-g'} could be strict.
\begin{example}\label{gapfree_example}
Let $G = C_6$ and $I = I(G)
=(x_1x_2,x_2x_3,x_3x_4,x_4x_5,x_5x_6,x_6x_1)\subset
k[x_1,\ldots,x_6]$. Then $(I^3:x_2x_3x_4x_5)=I+(x_1x_4,x_3x_6)$. Let
$G'$ be the graph associated to $(I^3:x_2x_3x_4x_5)$. It can be
easily seen that $\cochord(G)=\ab(G) = 2$ and $\cochord(G')=\ab(G') =
1$. By (\ref{ag_reg_chg}), $\reg(I)=3$ and $\reg(I^3:x_2x_3x_4x_5)=2$.
\end{example}

\section{Precise expressions for asymptotic regularity}\label{powers_regularity}
In this section, we apply Theorem \ref{upper_bound_reg} to obtain 
precise expressions for the regularity of powers of edge ideals of
various subclasses of bipartite graphs. 
We begin the study with some classes of graphs $G$ for which $\ab(G) =
\cochord(G)$. We then use
(\ref{eq:3}) to prove that $\reg(I(G)^s) = 2s+\ab(G)-1$ for such
graphs. 

\begin{corollary}\label{thm:ag=cg}
  Let $G$ be a bipartite graph. If
  \begin{enumerate}
	\item $G$ is unmixed; \footnote{In a personal
		communication, we have been informed that Banerjee and
		Mukundan have also obtained Corollary \ref{thm:ag=cg}(1).} 
	\item $G = W(H)$ for some bipartite graph $H$;
	\item $G$ is weakly chordal, ~ or
	\item If $G$ is $P_6$-free graph, 
  \end{enumerate}
then for all $s \geq 1, 
	  ~\reg(I( G)^s) = 2s + \ab(G) - 1$.
\end{corollary}
\begin{proof}
The assertions, (1), (2) and (3) follows from Remark \ref{ag=cg} and
(\ref{eq:3}). If $G$ is $P_6$-free, then so is $G'$. Therefore, the
result now follows from Theorem \ref{regularity_compa} and Corollary
\ref{banerjee-qn}.
\end{proof}
Observe that bipartite $P_5$-free  graphs are chordal bipartite.
Therefore by Corollary \ref{thm:ag=cg}(3), $\reg(I(G)^s)=2s+\ab(G)-1$ for
all $s \geq 1$. 
In general, for a bipartite $P_6$-free graph, it is not known whether the
equality $\ab(G) = \cochord(G)$ is true.
However, the previous result shows
that for $s \geq 1$, $\reg(I(G))^s) = 2s + \ab(G) -1 $.

Since forests are weakly chordal bipartite graphs, we derive, from
Corollary \ref{thm:ag=cg}(3), one of the main results of 
Beyarslan, H\`a and Trung:
\begin{corollary}\label{forest}\cite[Theorem 4.7]{selvi_ha} If $G$ is a forest,
  then  for all $s\geq 1$,
$$\reg(I(G)^s)=2s+\ab(G)-1.$$ 
\end{corollary}
The \textit{bipartite complement} of a bipartite graph
$G$ is the bipartite graph $G^{bc}$ on the same vertex set as $G$, $V(G^{bc})=X \cup Y$,
with $E(G^{bc}) = \{\{x, y\}| x \in X, y \in Y, \{x, y\} \notin E(G)\}$.
Below we make an observation on connected bipartite graphs $G$ with
$\reg(I(G)) = 3$. 

\begin{obs}\label{reg:3}If $G$ is a connected bipartite graph with
$\reg(I(G))=3$, then by \cite[Theorem 3.1]{FOGP}, $2 \leq \ab(G)$.
If $2 < \ba(G)$, then $G^{bc}$ has an induced cycle of length $6$,
which contradicts \cite[Theorem 3.1]{FOGP}. Therefore,
$\ab(G)=\cochord(G) = \ba(G)=2$.
\end{obs}

We now derive two results of 
Alilooee and Banerjee (\cite[Theorems 3.8, 3.9]{banerjee1}) as a
corollary:
\begin{corollary}\label{reg-3}
If $G$ is a connected bipartite graph with
$\reg(I(G)) = 3$, then 
\begin{enumerate}
  \item $\reg(I(G)^{s+1} : e_1\cdots e_s) \leq 3$ for any $s$-fold
	product $e_1\cdots e_s$;
  \item for all $s \geq 1,$ $\reg(I(G)^s) = 2s + 1. $
\end{enumerate}
\end{corollary}
\begin{proof}
The first assertion follows from Observation \ref{reg:3} and Proposition
\ref{reg-g-g'} and then second assertion follows
 from Observation \ref{reg:3} and (\ref{eq:4}).
\end{proof}


So far, we had been discussing about graphs $G$ for which
$\reg(I(G)^s) = 2s+\ab(G) - 1$ for all $s
\geq 1$. 
Now we produce some classes of graphs $G$ for which $\cochord(G) -
\ab(G)$ is arbitrarily large and hence $\reg(I(G)^s) - [2s+\ab(G) -
1]$ is also arbitrarily large. If $G$
is the disjoint union of $C_{n_1},\ldots,C_{n_m}$ and $k$ edges, 
then one can easily see that $$\ab(G)=k+\sum_{j=1}^m\left\lfloor
  \frac{n_j}{3}\right\rfloor,$$
\[
  \cochord(G) = \left\{\def\arraystretch{1.2}%
  \begin{array}{@{}c@{\quad}l@{}}
    k+\sum_{j=1}^m\lfloor  \frac{n_j}{3}\rfloor & \text{ if $n_1,\ldots,n_m \equiv
	  \{0,1\} ~(mod~3),$}\\
    k+m+\sum_{j=1}^m \lfloor\frac{n_j}{3}\rfloor & \text{ if $n_1,\ldots,n_m
	\equiv 2 ~(mod~3),$}\\
  \end{array}\right.
\]
and 
\[
  \reg(I(G)) = \left\{\def\arraystretch{1.2}%
  \begin{array}{@{}c@{\quad}l@{}}
    \ab(G)+1 & \text{ if $n_1,\ldots,n_m \equiv \{0,1\} ~(mod~3),$}\\
    \cochord(G)+1 & \text{ if $n_1,\ldots,n_m \equiv 2 ~(mod~3).$}\\
  \end{array}\right.
\]

We prove a precise expression for $\reg(I(G)^s)$ in this case.
We would like to thank Tai H\`{a} for indicating the following proof,
much simpler than the original one.
\begin{theorem}\label{upperbound} Let $G$ be the disjoint union of $C_{n_1},\ldots,C_{n_m}$ and $k$ edges, $k \geq 1$. Then for all $s \geq 1$,
\[
  \reg(I(G)^s) = \left\{\def\arraystretch{1.2}%
  \begin{array}{@{}c@{\quad}l@{}}
    2s+\ab(G)-1 & \text{ if $n_1,\ldots,n_m \equiv \{0,1\} ~(mod~3)$}\\
    2s+\cochord(G)-1 & \text{ if $n_1,\ldots,n_m \equiv 2 ~(mod~3)$}\\
  \end{array}\right.
\]
\end{theorem}
\begin{proof} 
Let $e_1,\ldots,e_k$ be disjoint edges. It follows from \cite[Theorem
7.6.28]{sean_thesis} and \cite[Theorem 5.2]{selvi_ha} that 
\begin{enumerate}
 \item if $n \equiv \{0,1\} (mod~3)$, then
   $\reg(I(C_n)^s)=2s+\ab(C_n)-1$, for all $s \geq 1$;
 \item if $n \equiv 2 (mod~3)$, then $\reg(I(C_n))=\cochord(C_n)+1$ and $\reg(I(C_n)^s)=2s+\ab(C_n)-1$, for all $s \geq 2$.
\end{enumerate}
Let $G_1=C_{n_1} \cup \{e_1\} \cup \cdots \cup \{e_k\}$. First we claim that, for $s \geq 1$
\[
  \reg(I(G_1)^s) = \left\{\def\arraystretch{1.2}%
  \begin{array}{@{}c@{\quad}l@{}}
    2s+\ab(G_1)-1 & \text{ if $n_1 \equiv \{0,1\} ~(mod~3)$}\\
    2s+\cochord(G_1)-1 & \text{ if $n_1 \equiv 2 ~(mod~3)$}\\
  \end{array}\right.
\]
We prove this by induction on $k$. Let $k=1$. Let $H_1=C_{n_1} \cup \{e_1\}$. 
By \cite[Lemma 2.5]{HTT}, we can prove the case $s=1$.
If $n_1 \equiv \{0,1\}(mod~3)$, then by \cite[Theorem 5.7]{nguyen_vu}, for $s \geq 2$, we have 
$$\reg(I(H_1)^s)=2s+\ab(H_1)-1.$$
If $n_1 \equiv 2 (mod~ 3)$, then by \cite[Proposition 2.7]{HTT}, we get
$$\reg(I(H_1)^2)=\reg(I(C_{n}))+\reg(I(e_1)^2)-1=\cochord(H_1)+3.$$
By \cite[Theorem 5.7]{nguyen_vu}, for $s \geq 3$, we have 
$$\reg(I(H_1)^s)=2s+\cochord(H_1)-1.$$ This completes the proof for $k = 1$.
Suppose $k > 1$. Let $G_1 = C_{n_1} \cup \{e_1, \ldots, e_k\}$, where
$e_1, \ldots, e_k$ are disjoint edges. Let $H = C_{n_1} \cup \{e_1,
\ldots, e_{k-1}\}$. By induction hypothesis, for $s \geq 1$,
\[
  \reg(I(H)^s) = \left\{\def\arraystretch{1.2}%
  \begin{array}{@{}c@{\quad}l@{}}
    2s+\ab(H)-1 & \text{ if $n_1 \equiv \{0,1\} (mod~3)$},\\
    2s+\cochord(H)-1 & \text{ if $n_1 \equiv 2 (mod~3)$}.\\
  \end{array}\right.
\]
 Since $G_1 = H \cup \{e_k\}$. By \cite[Lemma 2.5]{HTT} and \cite[Theorem 5.7]{nguyen_vu}, for $s \geq 1$, we have  
 \[
  \reg(I(G_1)^s) = \left\{\def\arraystretch{1.2}%
  \begin{array}{@{}c@{\quad}l@{}}
    2s+\ab(G_1)-1 & \text{if $n_1 \equiv \{0,1\} (mod~3)$},\\
    2s+\cochord(G_1)-1 & \text{if $n_1 \equiv 2 (mod~3)$}.\\
  \end{array}\right.
\]
Let $G_{m-1}= C_{n_1} \cup \cdots \cup C_{n_{m-1}} \cup \{e_1,\ldots,
e_k\}$. Then by induction on $m$, we get, for $s \geq 1$, 
\[
  \reg(I(G_{m-1})^s) = \left\{\def\arraystretch{1.2}%
  \begin{array}{@{}c@{\quad}l@{}}
    2s+\ab(G_{m-1})-1 & \text{ if $n_1,\ldots,n_{m-1} \equiv
	\{0,1\} (mod~3),$}\\
    2s+\cochord(G_{m-1})-1 & \text{ if $n_1,\ldots,n_{m-1} \equiv
  2 (mod~3).$}\\
  \end{array}\right.
\]
Let $G=G_{m-1} \cup C_{n_m}$. By \cite[Lemma 2.5]{HTT}, \cite[Proposition 2.7]{HTT} and \cite[Theorem 5.7]{nguyen_vu},
we have $s \geq 1$,
\[
  \reg(I(G)^s) = \left\{\def\arraystretch{1.2}%
  \begin{array}{@{}c@{\quad}l@{}}
    2s+\ab(G)-1 & \text{ if $n_1,\ldots,n_m \equiv \{0,1\} (mod~3),$}\\
    2s+\cochord(G)-1 & \text{ if $n_1,\ldots,n_m \equiv 2 (mod~3).$}\\
  \end{array}\right.
\]
This completes the proof of the theorem.
\end{proof}
It may be noted that $C_n$ is
not bipartite if $n = 2k+1$ for some $k$, but the upper bound in
Theorem \ref{upper_bound_reg} is still satisfied in this case.
Suppose $H \cong \coprod_{j=1}^m C_{3n_j+2} ~ \coprod ~ \coprod_{i=1}^k
e_i$, then $\ab(H)=k+ \sum_{j=1}^m n_j$ and $\cochord(H) = 
k+m+\sum_{j=1}^m n_j$.
By Theorem \ref{upperbound}, for $s \geq 1$, we have  \
$$\reg(I(H)^s)-[2s+\ab(G)-1]=m.$$
Woodroofe proved that if $H$ is an induced subgraph
of a graph $G$, then $\reg(I(G)) \geq k + m + \sum_{j=1}^m n_j +
1$, \cite[Corollary 11]{russ}. We obtain a similar bound for all the powers.

\begin{corollary}\label{russ_result} If a graph $G$ has an induced subgraph 
 $H \cong   \coprod_{j=1}^m C_{n_j}  \coprod  \coprod_{i=1}^k e_i,$
then
\[
  \reg(I(G)^s) \geq \left\{\def\arraystretch{1.2}%
  \begin{array}{@{}c@{\quad}l@{}}
    2s+(k+ \sum_{j=1}^m \lfloor \frac{n_j}{3}\rfloor)-1 & \text{ if $n_1,\ldots,n_m
	  \equiv \{0,1\} ~(mod~3),$}\\
    2s+(k+m+\sum_{j=1}^m \lfloor \frac{n_j}{3}\rfloor)-1 & \text{ if
	$n_1,\ldots,n_m \equiv 2 ~(mod~3).$}\\
  \end{array}\right.
\]
\end{corollary}
\begin{proof} Follows immediately from Theorem \ref{upperbound} and \cite[Corollary 4.3]{selvi_ha}.
\end{proof}

Note that if $n_j = 3k_j + 2$ for some $k_j \geq 1, ~ j = 1, \ldots,
m$, then $\reg(I(G)^s) \geq 2s + \cochord(H) - 1$. This is a much
improved lower bound in this class of graphs since $\cochord(H)$ could
be much larger than $\ab(G)$.
\begin{example}\label{ex_reg=cochord} 
Let $K_{1,n}$ be the complete bipartite graph with
partition $\{w\} \cup \{x_1,\ldots,x_n\}$. Let $n = k+m$. Let $G$ be
the graph obtained by attaching a pendant vertex each to $x_1, \ldots, x_k$
and identifying a vertex of $C_{2r_t}$ with $x_t$ for $k+1 \leq t \leq
n$, where $2r_t \equiv 2(mod~3)$.  Let $H$ be induced subgraph of $G$
on $V(G) \setminus \{w\}$. One can easily see that,
$\cochord(H)=\cochord(G)$.
Therefore it follows from Corollary \ref{russ_result} and Theorem
\ref{upper_bound_reg} that
for all $s \geq 1$,
$$\reg(I(G)^s)=2s+\cochord(G)-1.$$
It may also be noted that for this class of graphs $\reg(I(G)^s) -
[2s+\ab(G) - 1] = m$.
\end{example}

There are many classes of graphs $G$ for which the equality
$\reg(I(G)^s) = 2s + b$ has been established, where $b$ is some
combinatorial invariant associated with $G$. For all such results, the
constant term $b$ is either equal to $\ab(G) - 1$ or equal to
$\cochord(G) -1$. While $\ab(G) - 1$ is a lower bound for
$\reg(I(G)^s) - 2s$ for all graphs, $\cochord(G) - 1$ is an upper bound in the case
of bipartite graphs. Moreover, $\cochord(G) - \ab(G)$ can be
arbitrarily large.  We conclude our article with the following
question:

\begin{question}
Does there exist a graph $G$ such that for all $s \gg 0$ 
$$2s + \ab(G) - 1 < \reg(I(G)^s) < 2s + \cochord(G) - 1?$$ 
More generally, let $n = \cochord(G) - \ab(G)$ and $\mathcal{C}_n = \{G
  \mid \cochord(G) - \ab(G) = n \}$. For
  each $t \in \{0,1,\ldots, n\}$, does there exist $G_t \in
  \mathcal{C}_n$ such that $\reg(I(G_t)^s) = 2s + \ab(G) + t - 1?$
\end{question}

\vskip 2mm
\noindent
\textbf{Acknowledgement:} We would like to thank Adam Van Tuyl who pointed
us to the article \cite{selvi_ha}. We would also like to thank Arindam
Banerjee and Selvi Beyarslan for some useful discussions regarding the
materials discussed in this paper. We heavily used the commutative
algebra package, Macaulay 2, \cite{M2}, for verifying whichever
conjectures came to our mind. The third author is funded by
National Board for Higher Mathematics, India. 

\bibliographystyle{plain}  
\bibliography{refs_reg}
\end{document}